\documentclass[11pt]{amsart}
\usepackage{amsmath,amssymb,amsfonts,amsthm,amscd,indentfirst}
\usepackage{hyperref}\usepackage{xcolor}

\numberwithin{equation}{section}

\newtheorem{prop}{Proposition}[section]
\newtheorem{theo}[prop]{Theorem}
\newtheorem{lemm}[prop]{Lemma}

\newtheorem{rema}[prop]{Remark}

\def\and{\quad{\rm and}\quad}

\def\<{\langle}
\def\>{\rangle}

\title{On Weyl's embedding problem in Riemannian manifolds}
\author{
        Siyuan Lu
        }
\address{Department of Mathematics, Rutgers University, 110 Frelinghuysen Road, Piscataway, NJ 08854}
\email{siyuan.lu@math.rutgers.edu}
\thanks{Research of the author was supported in part by CSC fellowship and Schulich Graduate fellowship.}
\usepackage{graphicx}

\begin{document}

\begin{abstract}
We consider a priori estimates of Weyl's embedding problem of $(\mathbb{S}^2, g)$ in general $3$-dimensional Riemannian manifold $(N^3, \bar g)$. We establish interior $C^2$ estimate under natural geometric assumption. Together with a recent work by Li and Wang, we obtain an isometric embedding of $(\mathbb{S}^2,g)$ in Riemannian manifold. In addition, we reprove Weyl's embedding theorem in space form under the condition that $g\in C^2$ with $D^2g$ Dini continuous. 
\end{abstract}

\maketitle

\section{Introduction}

Weyl's embedding problem is a classic isometric embedding problem in differential geometry. It was raised by Weyl \cite{W} in 1916. The problem concerns the realization of $(\mathbb{S}^2,g)$ with positive Gauss curvature as a convex surface in $3$-dimensional Euclidean space. 

In the case that $g$ is analytic, it was solved by Lewy in \cite{L}. In the smooth case, Weyl's problem was fully resolved by Nirenberg in a milestone paper \cite{N}, under conditions that $g\in C^4$ and Gauss curvature $K_g>0$. In the degenerate case $K_g\ge 0$, Weyl's problem was studied by Guan and Li \cite{GL} and Hong and Zuily \cite{HZ}, see also \cite{I}.  

In the case of hyperbolic space $\mathbb{H}^3$, Weyl's problem was solved by Pogorelov \cite{P} under conditions that $g\in C^4$ and $K_g>-1$. The degenerate case $K_g\ge -1$ was studied by Lin and Wang \cite{WL}, see also early results by Chang and Xiao \cite{CX}. 

For general ambient spaces other than space form, Pogorelov established isometric embedding of a convex surface $(\mathbb{S}^2, g)$ into $(N,\bar{g})$ under certain conditions on sectional curvatures of the ambient space, see \cite{Pb} for details,  In a recent paper, Guan and the author \cite{GLu} obtained a global curvature bound if $(M^n, g)$ is isometrically embedded into warped product space and a class of ambient spaces. 

Instead of the global curvature estimates studied above, Heinz \cite{H1,H2} developed the interior $C^2$ estimate for the embedding. This allows him to solve the problem under the condition $g\in C^3$. This method was further modified by Schulz to relax the condition to $g\in C^{2,\alpha}$, see in \cite{S}. For general ambient spaces, Heinz \cite{H4} obtained a mean curvature estimate of a convex surface $(M^2,g)$ isometrically embedded into $(N,\bar{g})$ under conditions that $g\in C^3$ and the existence of a global convex function in ambient space, see also work by Dubrovin \cite{D}.

In this paper, we generalize Heinz's interior $C^2$ estimate to general Riemannian manifold. Note that a key difference for Heinz's interior $C^2$ estimate is that it's purely interior, does not depends on the position of the surface. This is totally different from global estimates above. The main contribution of our result is that the estimate will not depend on the position of the surface. This allows us to fulfill the closedness part of isometric embedding, see below for more details.

\medskip

Another motivation to study Weyl's problem comes from general relativity. This can be seen by the definitions of quasi-local mass of a bounded domain $(\Omega,\bar{g})$ in space time. Suppose $\Sigma=\partial \Omega$ is a topological sphere with positive Gauss curvature. In the time symmetric case, the Brown-York mass \cite{BY} is defined to be:
\begin{align*}
m_{BY}=\frac{1}{8\pi}\int_{\Sigma} (H_0-H) d\sigma,
\end{align*}
where $H$ is the mean curvature of $\Sigma$ in $(\Omega,\bar{g})$ and $H_0$ is the mean curvature of $\Sigma$ embedded into $\mathbb{R}^3$. While in the case of space time, the Liu-Yau quasi-local mass \cite{LY1, LY2} is defined to be
\begin{align*}
m_{LY}=\frac{1}{8\pi}\int_{\Sigma} (H_0-|H|) d\sigma,
\end{align*}
where $|H|$ is the Lorentzian norm of the mean curvature vector. The positivity of $m_{BY}$ was established by Shi and Tam \cite{ST2}, the positivity of $m_{LY}$ was established by Liu and Yau \cite{LY2}. By the above two definitions, we see that the solution of isometric embedding in $\mathbb{R}^3$ plays a crucial role. 

In the case that the Gauss curvature of $\Sigma$ ceases to be positive, Wang and Yau \cite{WY1} generalized the Liu-Yau quasi-local mass, using Pogorelov's solution to Weyl's problem in hyperbolic space $\mathbb H^3_{\kappa^2}$.  Their result was later furnished by Shi and Tam \cite{ST} to obtain the following inequality
\begin{align*}
\int_{\Sigma}(H_0-H)\cosh(\kappa r) d\sigma\geq 0,
\end{align*}
where $H_0$ is the mean curvature of $\Sigma$ embedded into $\mathbb{H}^3_{\kappa^2}$ and $r$ is the distance function in $\mathbb{H}^3_{\kappa^2}$. It turns out that this inequality plays an important role in our mean curvature estimate. Moreover, if we assume $\Sigma$ can be isometrically embedded into Schwarzschild manifold, then Miao and the author \cite{LuM} were able to establish a quasi-local type inequality. In space time case, Wang and Yau \cite{WY} further studied a new quasi-local mass using isometric embedding into Minkowski spacetime. 

\medskip

We consider a priori estimates of isometrically embedded surface $(\mathbb{S}^2,g)$ in $3$-dimensional Riemannian manifold $(N^3,\bar{g})$. Recall that the second order estimate in \cite{GLu} depends on the position of the surface in general. In the case that $(N^3,\bar{g})$ is a space form, it fulfills the a priori estimates as $C^0$ estimate follows directly from Bonnet's theorem. In general, since a priori we do not know the position of the surface, uniform $C^2$ estimate does not follow. On the other hand, in a recent work of Li and Wang \cite{GLW}, a uniform $C^0$ estimate was obtained for strictly convex surfaces. Thus to establish the full a priori estimates, we need to establish a new mean curvature estimate which is independent of the position of $(\mathbb{S}^2,g)$ in $(N,\bar{g})$. 

We adopt Heinz's interior $C^2$ estimate to solve the problem. There are two major advantages of our proof. The first is that our mean curvature estimate is indeed independent of the position of surface, thus fulfills the unifom $C^2$ estimates. The second is that we only require the metric $g,\bar{g}\in C^3$, while both $g, \bar{g}\in C^4$ is required in \cite{GLu}. 

A key observation in our proof is a bound of the total mean curvature. In Euclidean space, the total mean curvature can be bounded via the classic Minkowski formula, which was carried out by Heinz in \cite{H3}, see also \cite{S2}. Using the same idea, we are able to employ the analogous Minkowski formula to bound the total mean curvature in the case of space form. However, for general Riemannian manifold, in lack of conformal Killing vector field, the Minkowski formula no longer exists. In order to bound the total mean curvature, we use the positivity of hyperbolic version of quasi-local mass, which is proved by Wang and Yau \cite{WY1} and Shi and Tam \cite{ST}. The total mean curvature is thus bounded due to the estimates of isometric embedding in hyperbolic space.

\medskip

Before we state our theorem, let us fix some notations. Let $(\mathbb{S}^2,g)$ be an isometrically embedded surface in an ambient space $(N^3, \bar g)$. Denote $Ric$ and $\bar{Ric}$ the Ricci curvature tensors of $(\mathbb{S}^2,g)$  and ($N,\bar{g}$) respectively,  and denote $R=Tr(Ric)$  and $\bar{R}=Tr(\bar{Ric})$ the scalar curvatures of $(\mathbb{S}^2,g)$ and $(N,\bar{g})$ respectively. Let $\nu$ be the unit outer normal, and $h$ the corresponding second fundamental form, by Gauss equation, we have
\begin{align}\label{G1}
\det(h)=\frac{R-\bar{R}}{2}+\bar{Ric}(\nu,\nu).
\end{align}

We now state our main results.

\begin{theo}\label{thm-1}
Let $(N^3,\bar{g})$ be a 3-dimensional Riemannian manifold with possibly finite many boundary components, which are minimal surfaces. Suppose $\Sigma=(\mathbb{S}^2,g)$ is isometrically embedded into $(N^3,\bar{g})$, let $X$ be the embedding, assume that 
\begin{align}\label{ellip}
R(x)-\bar{R}(X(x))+2\bar{Ric}_{X(x)}(\nu,\nu)\geq C_0>0,
\end{align}
$\forall x\in \Sigma$. Assume further that $\bar{R}\geq -6\kappa^2$ for some costant $\kappa>0$. Then we have
\begin{align*}
\|D^2X\|_{C^{k,\mu}}\leq C,
\end{align*}
for any $0<\mu<1$, where $C$ depends only on $C_0$, $\kappa$, $\|g\|_{C^{k+3}}$ and $\|\bar{g}\|_{C^{k+3}}$, but not on the position of $\Sigma$ in $N$, here $k\geq 0$ is an integer.
\end{theo}

By Gauss equation (\ref{G1}), assumption (\ref{ellip}) implies $\det(h)\geq C_0$, which is the ellipticity condition in the view of partial differential equations. In particular, the mean curvature estimate corresponds to the case $k=0$, we refer to Theorem \ref{thm-curvature} in section 4 for detail.

\medskip

As mentioned above, the mean curvature estimate above is the key for closedness part of isometric embedding theorem. Together with the recent work by Li and Wang \cite{GLW} (the openness part of the isometric embedding problem), we can prove an isometric embedding in Riemannian manifold. In particular, we are able to prove an isometric embedding in Schwarzschild manifold outside the horizon. 

\begin{theo}\label{thm-isometric}
Suppose $(N^3,\bar{g})$ is a $C^3$ Riemannian manifold with possibly finite many boundary components, which are minimal surfaces. Suppose that $N$ is diffeomorphic to $\mathbb{R}^3$ minus a compact set and the scalar curvature $\bar{R}$ is bounded below, let $g$ be a $C^3$ metric on $\mathbb{S}^2$ such that
\begin{align}\label{curvature con thm}
R(x)\geq \bar{R}(X)-2\inf\{\bar{Ric}(X)(\xi,\xi)|\xi\in T_XN,|\xi|=1\}+ C_0,
\end{align}
$\forall x\in (\mathbb{S}^2,g)$ and $\forall X\in (N,\bar{g})$, where $C_0>0$ is an arbitrary constant.  Then $(\mathbb{S}^2,g)$ can be isometrically embedded into $(N,\bar{g})$ as a $C^{2,\mu}$ surface which is null-homologous to the boundary component, for any $0<\mu<1$. Morever, if $g,\bar{g}\in C^k$, $k\geq 3$, then the embedding is $C^{k-1,\mu}$. In particular, if both $g,\bar{g}$ are smooth, then the embedding is also smooth.
\end{theo}

\medskip

In view of Theorem \ref{thm-1} and Theorem \ref{thm-isometric}, we require $g\in C^3$ to establish $C^2$ embedding. Recall that in the case of Euclidean space, Schulz was able to establish the embedding under the condition that $g\in C^{2,\alpha}$. A natural question is what's the optimal regularity condition to ensure $C^2$ embedding. In general, a weaker regularity condition to ensure corresponding regularity is the Dini continuity for elliptic partial differential equations.  We say a function is Dini continuous if the moduli of continuity $\omega(r)$ satisfies
\begin{align*}
\int_0^1\frac{\omega(r)}{r}dr<\infty.
\end{align*}
It's easy to see that H\"older continuity implies Dini continuity. 

In this direction, we are able to weaken the condition to $g\in C^2$ with $D^2g$ Dini continuous in the case of space form.

\begin{theo}\label{thm-hyperbolic}
Let $g$ be a $C^2$ metric on $\mathbb{S}^2$, with $D^2g$ Dini continuous. Suppose the scalar curvature of $(\mathbb{S}^2,g)$ satisfies
\begin{align*}
R>2K,
\end{align*}
then $(\mathbb{S}^2,g)$ can be isometrically embedded into $N_K^3$ as a $C^2$ surface $\Sigma$, where $N_K^3$ denotes $3$-dimensional space form with constant sectional curvature $K$. 
\end{theo}

By $C^2$ surface, we mean that the embedding $X$ is $C^2$, with moduli of continuity of $D^2X$ under control. This is different from $C^{1,1}$ embedding. We remark that this theorem is new even in Euclidean space. 

\medskip

\medskip

The organization of the paper is as follows. In the next section, we state some lemmas concerning the classic isothermal parameters, quasi-local mass and Heinz's interior $C^2$ estimate. In section 3, we consider the isometric embedding in space form. In section 4, we consider the a priori estimates of the embedding and prove Theorem \ref{thm-1}. We study the isometric embedding into Riemannian manifold in section 5. The proof of technical Lemma \ref{Heinz-Lewy Dini} will be given in the Appendix.

\section{Preliminary}
Let us list some basic formulas for isometrically immersed surface $(M^2,g)$ in an ambient Riemannian manifold $(N^3, \bar g)$. Denote $R_{ijkl}$ and $\bar{R}_{abcd}$ to be the Riemannian curvatures of $M$ and $N$ respectively. For a fixed local frame $(e_1, e_2)$ on $M$, let $\nu$ be a normal vector field of $M$, and let $h=(h_{ij})$ be the second fundamental form of $M$ with respect to $\nu$. We have the Gauss equation and Codazzi equation,
\begin{align}\label{Gauss}
R_{ijkl}=\bar{R}_{ijkl}+h_{ik}h_{jl}-h_{il}h_{jk}, \quad (Gauss)
\end{align}
\begin{align}\label{Codazzi}
\nabla_k h_{ij}=\nabla_j h_{ik}+\bar{R}_{\nu ijk} .\quad (Codazzi)
\end{align}
We use the convention that $R_{ijij}$ denotes the sectional curvature.

Taking trace of (\ref{Gauss}), we have
\begin{align}\label{G}
\det(h^i_j)=\frac{R-\bar{R}}{2}+\bar{Ric}(\nu,\nu),
\end{align}
where $h^i_j=g^{ik}h_{kj}$ is the Weingarten tensor.

\medskip

We first state a classic result about the existence of isothermal parameters of surface by Hartman and Wintner, see in \cite{HW}.
\begin{lemm}\label{isothermal}
Let $\Sigma$ be a surface and $B_r$ a domain in $\Sigma$. Let $(l_{ij})$ be a $2\times 2$, positive definite symmetric matrix in $B_r$, such that $l_{ij}$ is Dini continuous in $B_r$. Then there exist isothermal parameters $(u,v)$ such that 
\begin{align*}
ds^2=l_{11}dx^2+2l_{12}dxdy+l_{22}dy^2=\Lambda(du^2+dv^2),\quad \Lambda\neq 0,
\end{align*}
in $B_r$ for $r$ sufficiently small. 
\end{lemm}

\medskip

We also need the following lemma concerning the positivity of quasi-local mass, which was proved by Wang and Yau \cite{WY1} and Shi and Tam \cite{ST}.

\begin{lemm}\label{quasi}
Let $(\Omega,\bar{g})$ be a $3$-dimensional Riemannian manifold with scalar curvature $\bar{R}\geq -6\kappa^2$ for some $\kappa>0$. Assume that $\Sigma=\partial \Omega$ is a topological sphere with scalar curvature $R> -2\kappa^2$ and positive mean curvature $H$. Then $\Sigma$ can be isometrically embedded into $\mathbb{H}_{\kappa^2}^3$, let $H_0$ be the corresponding mean curvature, then we have
\begin{align*}
\int_\Sigma \left(H_0-H\right)\cosh (\kappa r) d\sigma \geq 0,
\end{align*}
where $r$ is the distance function from the origin in $\mathbb{H}_{\kappa^2}^3$.
\end{lemm}

In the case that $(N^3,\bar{g})$ has horizons, we need extend Lemma \ref{quasi} to the case that $(\Omega, \bar{g})$ has more than one boundary components. This is proved in a recent work by Mantoulidis and Miao \cite{MM} based on the work of Wang and Yau \cite{WY1} and Shi and Tam \cite{ST}.

\begin{lemm}\label{quasi-2}
Let $(\Omega,\bar{g})$ be a $3$-dimensional Riemannian manifold with scalar curvature $\bar{R}\geq -6\kappa^2$ for some $\kappa>0$. Assume that $\Sigma$ is one conneted component of the boundary of $(\Omega,\bar{g})$, which is a topological sphere with scalar curvature $R> -2\kappa^2$ and positive mean curvature $H$. Assume further that all other connected components of the boundary of $(\Omega,\bar{g})$ are minimal surfaces. Then $\Sigma$ can be isometrically embedded into $\mathbb{H}_{\kappa^2}^3$, let $H_0$ be the corresponding mean curvature, then we have
\begin{align*}
\int_\Sigma \left(H_0-H\right)\cosh (\kappa r) d\sigma \geq 0,
\end{align*}
where $r$ is the distance function from the origin in $\mathbb{H}_{\kappa^2}^3$.
\end{lemm}

\begin{rema}
We remark that the assumption in \cite{MM} is less restrictive than Lemma \ref{quasi-2}, in fact as long as the mean curvature of all other boundary component of $(\Omega,\bar{g})$ w.r.t. the outer normal is greater or equal than $-2\kappa$, the conclusion still holds. In our case, Lemma \ref{quasi-2} is suffice as we already assume all boundary component of $(N,\bar{g})$ are minimal surfaces.
\end{rema}
\medskip

We attack the problem using Heinz's interior $C^2$ estimate. The procedure involves two steps. The first step is to establish a suitable equation system and to transform the problem to the estimate of the elliptic system. The second to to obtain a lower bound of the Jacobian $J(x,y)$ of the map. The following lemma is in fact the second step, see Theorem 8.3.2 in \cite{S}.

Denote $w=(u,v)\in \mathbb{R}^2,z=(x,y)=z(w)\in \Omega$. Let $x,y\in W^{1,2}(\mathbb{R}^2)$ be solution of the following elliptic system of Heinz-Lewy type
\begin{align}\label{Heinz-Lewy}
Lx&=h_1(z)|Dx|^2+h_2(z)Dx\cdot Dy+h_3(z)|Dy|^2+h_4(z)Dx\wedge Dy,\\\nonumber
Ly&=\tilde{h}_1(z)|Dx|^2+\tilde{h}_2(z)Dx\cdot Dy+\tilde{h}_3(z)|Dy|^2+\tilde{h}_4(z)Dx\wedge Dy,
\end{align}
where
\begin{align*}
L=-\frac{1}{a(z,w)}D_\alpha\left(a(z,w)D_\alpha\right),\quad Dx\wedge Dy=x_uy_v-x_vy_u.
\end{align*}

Assumption: 
\begin{enumerate}
\item  $a\in C^\mu (\Omega\times \mathbb{R}^2)$, such that
\begin{align*}
0<\lambda \leq a(z,w)\leq \Lambda <+\infty,\quad \|a\|_{C^\mu} \leq M_1,
\end{align*}
\item 
\begin{align*}
|h_1(z)|,\cdots,|\tilde{h}_4(z)|\leq M_0,
\end{align*}
\item \begin{align*}
\tilde{h}_1(z)=0,\quad h_1(z)-\tilde{h}_2(z)=0,\quad h_2(z)-\tilde{h}_3(z)=0,\quad h_3(z)=0,
\end{align*}
\item  $a$ is Dini continuous with moduli of continuity $\omega$, such that
\begin{align*}
0<\lambda \leq a(z,w)\leq \Lambda <+\infty,
\end{align*}
\item 
\begin{align*}
h_1(z)=\cdots=\tilde{h}_4(z)=0.
\end{align*}
\end{enumerate}

We remark that assumption (3) is a structure assumption on the equation system to obtain the lower bound of Jacobian $J(x,y)$, assumption (5) implies that $Lx=Ly=0$.

\begin{lemm}\label{Heinz-Lewy theorem}
Let $z(w)$ be a homeomorphism from the closed unit disk $\bar{B}$ onto itself of class $W^{1,2}$ solving the system (\ref{Heinz-Lewy}). Suppose that assumptions (1,2,3) are satisfied, assume that $z(0)=0$ and
\begin{align*}
\int_{B}|Dz|^2dudv\leq M.
\end{align*}
Then $z\in C^{1,\mu}_{loc}(B)\cap C^0(\bar{B})$ and the Jacobian $J(x,y)=x_uy_v-x_vy_u$ does not vanish in $B$. Moreover, for all $B_R$, $0<R<1$, we have
\begin{align*}
\|z\|_{C^{1,\mu}(B_R)}\leq C,\\
J(x,y)\geq c>0,
\end{align*}
where $C,c$ depends on $\mu,\lambda,\Lambda, M_1, M_0,M$ and $R$.

\end{lemm}

In order to prove isometric embedding theorem in space form for Dini continuity case, we need to extend the above lemma to the case that $a$ is only Dini continuous. This lemma will be proved in appendix.

\begin{lemm}\label{Heinz-Lewy Dini}
Let $z(w)$ be a homeomorphism from the closed unit disc $\bar{B}$ onto itself of class $W^{1,2}$ solving the system (\ref{Heinz-Lewy}). Suppose that assumptions (4,5) are satisfied, assume that $z(0)=0$ and
\begin{align*}
\int_{B}|Dz|^2dudv\leq M.
\end{align*}
Then $z\in C^{1}_{loc}(B)\cap C^0(\bar{B})$ and the Jacobian $J(x,y)=x_uy_v-x_vy_u$ does not vanish in $B$. Moreover, for all $B_R$, $0<R<1$, we have
\begin{align*}
\|z\|_{C^{1}(B_R)}\leq C,\\
J(x,y)\geq c>0,
\end{align*}
where $C,c$ depends on $\lambda,\Lambda, M,\omega$ and $R$. In particular, the moduli of continuity of $Dz$ is under control.

\end{lemm}

\section{Isometric embedding into space form}
In this section, we prove isometric embedding into space form using Heinz's type interior $C^2$ estimate, under the condition that $g\in C^2$ with $D^2g$ Dini continuous.

We first state a proposition.
\begin{prop}\label{prop-1}
Let $\Sigma=(\mathbb{S}^2, g)$ be a $C^2$ surface isometrically embedded into space form $N_K^3$, such that the second order derivative is Dini continuous. For any point $(x_0,y_0)\in \Sigma$, let $B_{r_0}(x_0,y_0)$ be a geodesic disk in $\Sigma$,  where $r_0$ is a fixed small constant. Suppose that
\begin{align*}
D=\frac{R-2K}{2}\det (g_{ij})\geq C_0>0,\quad \forall (x,y)\in  \bar{B}_{r_0},
\end{align*}
and
\begin{align*}
\int_\Sigma Hd\sigma \leq M.
\end{align*}
Then there exists a homeomorphism $(x,y)=(x(u,v),y(u,v))$ from $\bar{B}=\{u^2+v^2\leq 1\}$ onto $\bar{B}_{r_0}$ of class $C^1_{loc}(B)\cap C^0(\bar{B})$ with $x(0)=x_0,y(0)=y_0$ such that
\begin{align}
\frac{h_{11}}{\sqrt{D}}&=\frac{y_u^2+y_v^2}{J(x,y)},\\\nonumber
-\frac{h_{12}}{\sqrt{D}}&=\frac{x_uy_u+x_vy_v}{J(x,y)},\\\nonumber
\frac{h_{22}}{\sqrt{D}}&=\frac{x_u^2+x_v^2}{J(x,y)}.
\end{align} 
Further more,
\begin{align}\label{divergence operator}
Lx=0,\quad Ly=0,
\end{align}
where
\begin{align*}
L=\frac{\partial}{\partial u}\left(\sqrt{D}\frac{\partial}{\partial u}\right)+\frac{\partial}{\partial v}\left(\sqrt{D}\frac{\partial}{\partial v}\right).
\end{align*}

\end{prop}

\begin{proof}
We first prove the existence of isothermal parameters $(u,v)$ in $B_{r_0}$.

By (\ref{G}), we have
\begin{align*}
\det(h_{ij})=\det(h^i_j)\det(g_{ij})=\frac{R-2K}{2}\det(g_{ij})=D\geq C_0.
\end{align*}
Thus we can consider
\begin{align*}
ds^2=h_{11}dx^2+2h_{12}dxdy+h_{22}dy^2,
\end{align*}
as a positive metric in $B_{r_0}$.

Since $\Sigma$ is a $C^2$ surface with $D^2X$ Dini continuous, thus $h_{ij}$ is Dini continuous. By Lemma \ref{isothermal}, we have
\begin{align}\label{metric-1}
ds^2=\Lambda (du^2+dv^2),\quad \Lambda\neq 0,
\end{align}
in each $B_r\subset B_{r_0}$.

This is the local existence of isothermal parameters. We now want to apply uniformization theorem to prove the existence of $(u,v)$ in $B_{r_0}$. To do that, we only need to check the transformation map is holomorphic.

Let $w_1=(u_1,v_1)$, $w_2=(u_2,v_2)$ be two isothermal parameters. In the intersection area, we have
\begin{align*}
ds^2=\Lambda_1(du_1^2+dv_1^2)=\Lambda_2 (du_2^2+dv_2^2).
\end{align*}

By simple calculation, we have
\begin{align*}
&{u_1}^2_{u_2}+{v_1}^2_{u_2}={u_1}^2_{v_2}+{v_1}^2_{v_2},\\
&{u_1}_{u_2}{u_1}_{v_2}+{v_1}_{u_2}{v_1}_{v_2}=0.
\end{align*}
This is in fact the standard Cauchy-Riemann system
\begin{align*}
{u_1}_{u_2}={v_1}_{v_2},\quad {u_1}_{v_2}=-{v_1}_{u_1},
\end{align*}
i.e.
\begin{align*}
{w_1}_{\bar{w}_2}=0.
\end{align*}

Thus we have a homeomorphism $(x,y)=(x(u,v),y(u,v))$ from $\bar{B}=\{u^2+v^2\leq 1\}$ onto $\bar{B}_{r_0}$ with $x(0)=x_0,y(0)=y_0$.

By isothermal parameters (\ref{metric-1}), we have 
\begin{align}\label{inverse-1}
&\sqrt{D}x_u=h_{12}x_v+h_{22}y_v,\\\nonumber
&\sqrt{D}x_v=-h_{12}x_u-h_{22}y_u,
\end{align}
and
\begin{align}
&\sqrt{D}y_u=-h_{11}x_v-h_{12}y_v,\\\nonumber
&\sqrt{D}y_v=h_{11}x_u+h_{12}y_u.
\end{align}
 
The equivalent conformal relations are
\begin{align}\label{conformal-1}
\frac{h_{11}}{\sqrt{D}}&=\frac{y_u^2+y_v^2}{J(x,y)},\\\nonumber
-\frac{h_{12}}{\sqrt{D}}&=\frac{x_uy_u+x_vy_v}{J(x,y)},\\\nonumber
\frac{h_{22}}{\sqrt{D}}&=\frac{x_u^2+x_v^2}{J(x,y)}.
\end{align} 

We now derive the relation (\ref{divergence operator}). By (\ref{inverse-1}), we deduce that
\begin{align*}
Lx&=(h_{12})_ux_v-(h_{12})_vx_u+(h_{22})_uy_v-(h_{22})_vy_u\\
&=\left((h_{22})_x-(h_{12})_y\right)J(x,y).
\end{align*}
By Codazzi equation (\ref{Codazzi})
\begin{align}
\nabla_1 h_{22}-\nabla_2 h_{12}=0.
\end{align}
thus
\begin{align*}
Lx=0.
\end{align*}

Similarly, we have
\begin{align*}
Ly=0.
\end{align*}

By (\ref{conformal-1}), we have
\begin{align*}
\int_B\left(|{Dx}|^2+|{Dy}|^2\right)dudv=\int_{B_{r_0}}\frac{h_{11}+h_{22}}{\sqrt{D}}dxdy\leq C\int_\Sigma H d\sigma \leq C,
\end{align*}
as $g$ is continuous and $D\geq C_0$ on $B_{r_0}$.

Together with (\ref{divergence operator}), Lemma 1.6.3 in \cite{S} and recent theorem in \cite{Li} by Y.Y. Li(see Lemma \ref{Li} below), we conclude that
\begin{align*}
x,y\in C^1_{loc}(B)\cap C^0(\bar{B}).
\end{align*}

The propsition is now proved.

\end{proof}

\medskip

We are now in position to prove Theorem \ref{thm-hyperbolic}.

\begin{proof}
We first need to verify the condition
\begin{align*}
\int_\Sigma Hd\sigma\leq M.
\end{align*}
Since space form are of warped product structures, we can write the metric $\bar{g}$ as follows
\begin{align*}
\bar{g}=dr^2+\phi^2(r)g_{\mathbb{S}^2},
\end{align*}
where $\phi(r)=r$ if $K=0$, $\phi(r)=\frac{\sinh(\kappa r)}{\kappa}$ if $K=-\kappa^2$, $\phi(r)=\frac{\sin(\kappa r)}{\kappa}$ if $K=\kappa^2$ and $g_{\mathbb{S}^2}$ is the standard metric on $\mathbb{S}^2$. Note that $g_{\mathbb{S}^2}$ is different from the metric $g$ on $\Sigma$.

Since warped product spaces equip with a conformal Killing vector field, we have the following formulas, see for example in \cite{GLi}.
\begin{align}\label{conformal killing}
\Phi_{ij}=\phi^\prime g_{ij}-h_{ij}\left\langle \phi \frac{\partial}{\partial r},\nu\right\rangle,
\end{align}
where $\Phi=\int_0^r \phi(\rho)d\rho$ and $\nu$ is the unit outer normal of $\Sigma$.

Take trace of (\ref{conformal killing}) with $\frac{\partial \det(h_{ij})}{\partial h_{ij}}$, and integrate on $\Sigma$, we have
\begin{align*}
\int_\Sigma \frac{\partial \det(h_{ij})}{\partial h_{ij}}\Phi_{ij}d\sigma=\int_\Sigma H\phi^\prime d\sigma-2\int_\Sigma \det(h_{ij}) \left\langle \phi \frac{\partial}{\partial r},\nu\right\rangle d\sigma.
\end{align*}
By Gauss equation (\ref{G}) and Codazzi equation (\ref{Codazzi}) , we have
\begin{align*}
\int_\Sigma H\phi^\prime d\sigma=\int_\Sigma (R+2\kappa^2)\left\langle \phi \frac{\partial}{\partial r},\nu\right\rangle d\sigma.
\end{align*}

For $K=0$, we have $\phi^\prime=1$. For $K=-\kappa^2$, we have $\phi^\prime=\cosh(\kappa r)\geq 1$. For $K=\kappa^2$, by the Rauch comparison theorem, the diameter of $\Sigma$ is strictly less then $\frac{\pi}{\kappa}$, by a proper choice of origin, we have $\phi^\prime=\cos(\kappa r)\geq C$. 

To sum up, we have
\begin{align}\label{wyst inequality}
\int_\Sigma Hd\sigma\leq \frac{1}{C}\int_\Sigma H\phi^\prime d\sigma= \frac{1}{C}\int_\Sigma ( R+2\kappa^2)\left\langle \phi \frac{\partial}{\partial r},\nu\right\rangle d\sigma\leq M.
\end{align}
Here we use the fact that $H>0$ for convex surfaces and the diameter of $\Sigma$ is bounded by the $\|g\|_{C^2}$.

We can now apply Lemma \ref{Heinz-Lewy Dini}, together Proposition \ref{prop-1}, we have
\begin{align*}
|h_{ij}|\leq C,
\end{align*}
for $z=(x,y)\in B_{r_0/2}$, where $C$ depends only on $C_0$, $\kappa$, $\|g\|_{C^2}$ and $\omega$.

Moreover, we have for $h_{ij}(z)=h_{ij}(w(z))$
\begin{align}\label{C^2 alpha}
|h_{ij}(w_1)-h_{ij}(w_2)|\leq C\beta(|w_1-w_2|),
\end{align}
for $z_1,z_2\in B_{r_0/2}$, where $w_1=(u_1,v_2), w_2=(u_2,v_2)$, constant $C$ and function $\beta$ depends only on $C_0$, $\kappa$, $\|g\|_{C^2}$ and $\omega$.

By Propsition \ref{prop-1}, $x,y\in C^0(\bar{B})$ and map boundary to boundary, this implies there exists $0<\rho<1$, such that if $(x,y)\in B_{r_0/2}$, then $u^2+v^2<\rho^2$.

By Taylor expansion
\begin{align*}
|z_1-z_2|=|J(x,y)(\tilde{w})||w_1-w_2|\geq c|w_1-w_2|,
\end{align*}
for $z_1,z_2\in B_{r_0/2}$. Together with (\ref{C^2 alpha}), we have
\begin{align*}
|h_{ij}(z_1)-h_{ij}(z_2)|\leq C\beta(|z_1-z_2|)
\end{align*}
for $z_1,z_2\in B_{r_0/2}$.

In view of (\ref{conformal killing}), we have
\begin{align*}
\|\Phi\|_{C^2}\leq C
\end{align*}
for $(x,y)\in B_{r_0/2}$, with moduli of continuity of $D^2\Phi$ under control.

The full a priori estimate now follows by the fact that $\Sigma$ can be covered by finite patches of $B_{r_0/2}$.

For the case $K\leq 0$, the theorem follows by Pogorelov's isometric embedding into hyperbolic space. For the case $K>0$, we can first establish isometric embedding using the openness theorem by Li and Wang \cite{GLW}, i.e. Theorem \ref{openness theorem} in section 5 and the normalized Ricci flow. The theorem now follows from the a priori estimate established above.

\end{proof}

\section{A priori estimate of the embedding}
In this section, we consider the a priori estimate of the embedding. We first state a proposition, which is essentially contained in \cite{H4}, too. For the sake of completeness, we give the proof here.
\begin{prop}\label{key prop}
Let $\Sigma=(\mathbb{S}^2, g)$ be a $C^3$ surface isometrically embedded into a $C^3$ ambient space $(N^3,\bar{g})$, for any point $(x_0,y_0)\in \Sigma$, let $B_{r_0}(x_0,y_0)$ be a geodesic disk in $\Sigma$, where $r_0$ is a fixed small constant. Suppose that
\begin{align*}
D=\left(\frac{R-\bar{R}}{2}+\bar{Ric}(\nu,\nu)\right)\det(g_{ij})\geq C_0>0,\quad \forall (x,y)\in  \bar{B}_{r_0},
\end{align*}
and
\begin{align*}
\int_\Sigma Hd\mu \leq M.
\end{align*}
Then there exists a homeomorphism $(x,y)=(x(u,v),y(u,v))$ from $\bar{B}=\{u^2+v^2\leq 1\}$ onto $\bar{B}_{r_0}$ of class $C^{1,\mu}_{loc}(B)\cap C^0(\bar{B})$ for any $0<\mu<1$ with $x(0)=x_0,y(0)=y_0$ such that
\begin{align}\label{conformal-3}
\frac{h_{11}}{\sqrt{D}}&=\frac{y_u^2+y_v^2}{J(x,y)},\\\nonumber
-\frac{h_{12}}{\sqrt{D}}&=\frac{x_uy_u+x_vy_v}{J(x,y)},\\\nonumber
\frac{h_{22}}{\sqrt{D}}&=\frac{x_u^2+x_v^2}{J(x,y)}.
\end{align} 
Further more,
\begin{align}\label{laplace}
\Delta x&=h_1(z)|Dx|^2+h_2(z)Dx\cdot Dy+h_3(z)|Dy|^2+h_4(z)Dx\wedge Dy,\\\nonumber
\Delta y&=\tilde{h}_1(z)|Dx|^2+\tilde{h}_2(z)Dx\cdot Dy+\tilde{h}_3(z)|Dy|^2+\tilde{h}_4(z)Dx\wedge Dy,
\end{align}
where
\begin{align*}
\Delta=\frac{\partial^2}{\partial u^2}+\frac{\partial^2}{\partial v^2}
\end{align*}
is the flat Laplacian of $(u,v)$, and
\begin{align*}
h_1&=-\frac{1}{2D}\left(\left(\frac{R_1-\bar{R}_\alpha X^\alpha_1}{2}+\bar{Ric}_\alpha(\nu,\nu) X^\alpha_1\right)\det(g_{\alpha\beta})+\left(\frac{R-\bar{R}}{2}+\bar{Ric}(\nu,\nu)\right)\det(g_{\alpha\beta})_1\right),\\
h_2&=-\frac{1}{2D}\left(\left(\frac{R_2-\bar{R}_\alpha X^\alpha_2}{2}+\bar{Ric}_\alpha(\nu,\nu) X^\alpha_2\right)\det(g_{\alpha\beta})+\left(\frac{R-\bar{R}}{2}+\bar{Ric}(\nu,\nu)\right)\det(g_{\alpha\beta})_2\right),\\
h_3&=0,\\
h_4&=\frac{\bar{R}_{\nu 221}-g^{11}\bar{Ric}(e_1,\nu)\det(g_{\alpha\beta})}{\sqrt{D}},
\end{align*}
\begin{align*}
\tilde{h}_1&=0,\\
\tilde{h}_2&=-\frac{1}{2D}\left(\left(\frac{R_1-\bar{R}_\alpha X^\alpha_1}{2}+\bar{Ric}_\alpha(\nu,\nu) X^\alpha_1\right)\det(g_{\alpha\beta})+\left(\frac{R-\bar{R}}{2}+\bar{Ric}(\nu,\nu)\right)\det(g_{\alpha\beta})_1\right),\\
\tilde{h}_3&=-\frac{1}{2D}\left(\left(\frac{R_2-\bar{R}_\alpha X^\alpha_2}{2}+\bar{Ric}_\alpha(\nu,\nu) X^\alpha_2\right)\det(g_{\alpha\beta})+\left(\frac{R-\bar{R}}{2}+\bar{Ric}(\nu,\nu)\right)\det(g_{\alpha\beta})_2\right),\\
\tilde{h}_4&=\frac{\bar{R}_{\nu 112}-g^{22}\bar{Ric}(e_2,\nu)\det(g_{\alpha\beta})}{\sqrt{D}}.
\end{align*}

\end{prop}

\begin{proof}
By (\ref{G}), we have
\begin{align*}
\det(h_{ij})=\left(\frac{R-\bar{R}}{2}+\bar{Ric}(\nu,\nu)\right)\det(g_{ij})=D\geq C_0.
\end{align*}
Thus we can consider
\begin{align*}
ds^2=h_{11}dx^2+2h_{12}dxdy+h_{22}dy^2,
\end{align*}
as a positive metric on $B_{r_0}$.

Similarly to Proposition \ref{prop-1}, there exists homeomorphism $(x,y)=(x(u,v),y(u,v))$ from $\bar{B}=\{u^2+v^2\leq 1\}$ onto $\bar{B}_{r_0}$ of class $C^{1,\mu}_{loc}(B)\cap C^0(\bar{B})$ for any $0<\mu<1$ with $x(0)=x_0,y(0)=y_0$ such that
\begin{align*}
ds^2=\Lambda (du^2+dv^2).
\end{align*}
Note that instead of using the Theorem in \cite{Li} by Y.Y. Li, we may use Theorem 2.4.4 in \cite{S} to obtain the $C^{1,\mu}_{loc}$ estimate.

We only need to verify the relation (\ref{laplace}).

By this isothermal parameters, we have
\begin{align}\label{inverse}
&\sqrt{D}x_u=h_{12}x_v+h_{22}y_v,\\\nonumber
&\sqrt{D}x_v=-h_{12}x_u-h_{22}y_u,
\end{align}
and
\begin{align}
&\sqrt{D}y_u=-h_{11}x_v-h_{12}y_v,\\\nonumber
&\sqrt{D}y_v=h_{11}x_u+h_{12}y_u.
\end{align}
 
By (\ref{inverse}), we deduce that
\begin{align*}
\Delta x&=\left(\frac{h_{12}x_v+h_{22}y_v}{\sqrt{D}}\right)_u+\left(\frac{-h_{12}x_u-h_{22}y_u}{\sqrt{D}}\right)_v\\
&=\left(\frac{h_{12}}{\sqrt{D}}\right)_yy_ux_v+\left(\frac{h_{22}}{\sqrt{D}}\right)_xx_uy_v-\left(\frac{h_{12}}{\sqrt{D}}\right)_yy_vx_u-\left(\frac{h_{22}}{\sqrt{D}}\right)_xx_vy_u\\
&=\left(\left(\frac{h_{22}}{\sqrt{D}}\right)_x-\left(\frac{h_{12}}{\sqrt{D}}\right)_y\right)J(x,y).
\end{align*}

Similarly,
\begin{align*}
\Delta y=\left(\left(\frac{h_{11}}{\sqrt{D}}\right)_y-\left(\frac{h_{12}}{\sqrt{D}}\right)_x\right)J(x,y).
\end{align*}

Thus
\begin{align}\label{x}
\Delta x&= \frac{\nabla_1 h_{22}-\nabla_2 h_{12}}{\sqrt{D}}J(x,y)-\frac{h_{22}D_1-h_{12}D_2}{2D^{3/2}}J(x,y).
\end{align}

By Codazzi equation (\ref{Codazzi})
\begin{align}\label{eq-1}
\nabla_1 h_{22}-\nabla_2 h_{12}=\bar{R}_{\nu 221}.
\end{align}

since
\begin{align*}
D=\left(\frac{R-\bar{R}}{2}+\bar{Ric}(\nu,\nu)\right)\det(g_{ij}),
\end{align*}

we have
\begin{align*}
D_i=&\left(\frac{R_i-\bar{R}_\alpha X^\alpha_i}{2}+\bar{Ric}_\alpha(\nu,\nu) X^\alpha_i+2\bar{Ric}(e_j,\nu)h^j_i\right)\det(g_{\alpha\beta})\\
&+\left(\frac{R-\bar{R}}{2}+\bar{Ric}(\nu,\nu)\right)\det(g_{\alpha\beta})_i.
\end{align*}

Thus
\begin{align}\label{eq-2}
h_{22}D_1-h_{12}D_2=&h_{22}\left(\frac{R_1-\bar{R}_\alpha X^\alpha_1}{2}+\bar{Ric}_\alpha(\nu,\nu) X^\alpha_1+2\bar{Ric}(e_j,\nu)h^j_1\right)\det(g_{\alpha\beta})\\\nonumber
&+h_{22}\left(\frac{R-\bar{R}}{2}+\bar{Ric}(\nu,\nu)\right)\det(g_{\alpha\beta})_1\\\nonumber
&-h_{12}\left(\frac{R_2-\bar{R}_\alpha X^\alpha_2}{2}+\bar{Ric}_\alpha(\nu,\nu) X^\alpha_2+2\bar{Ric}(e_j,\nu)h^j_2\right)\det(g_{\alpha\beta})\\\nonumber
&-h_{12}\left(\frac{R-\bar{R}}{2}+\bar{Ric}(\nu,\nu)\right)\det(g_{\alpha\beta})_2\\\nonumber
=&\frac{|Dx|^2\sqrt{D}}{J(x,y)}\left(\frac{R_1-\bar{R}_\alpha X^\alpha_1}{2}+\bar{Ric}_\alpha(\nu,\nu) X^\alpha_1\right)\det(g_{\alpha\beta})\\\nonumber
&+\frac{|Dx|^2\sqrt{D}}{J(x,y)}\left(\frac{R-\bar{R}}{2}+\bar{Ric}(\nu,\nu)\right)\det(g_{\alpha\beta})_1\\\nonumber
&+\frac{Dx\cdot Dy \sqrt{D}}{J(x,y)}\left(\frac{R_2-\bar{R}_\alpha X^\alpha_2}{2}+\bar{Ric}_\alpha(\nu,\nu) X^\alpha_2\right)\det(g_{\alpha\beta})\\\nonumber
&+\frac{Dx\cdot Dy \sqrt{D}}{J(x,y)}\left(\frac{R-\bar{R}}{2}+\bar{Ric}(\nu,\nu)\right)\det(g_{\alpha\beta})_2\\\nonumber
&+2\left( h_{22}h^1_1-h_{12}h^1_2\right)\bar{Ric}(e_1,\nu)\det(g_{\alpha\beta}).
\end{align}
we have used (\ref{conformal-3}) in the second equality.

Plug (\ref{eq-1}) and (\ref{eq-2}) into (\ref{x}), we have
\begin{align*}
h_1&=-\frac{1}{2D}\left(\left(\frac{R_1-\bar{R}_\alpha X^\alpha_1}{2}+\bar{Ric}_\alpha(\nu,\nu) X^\alpha_1\right)\det(g_{\alpha\beta})+\left(\frac{R-\bar{R}}{2}+\bar{Ric}(\nu,\nu)\right)\det(g_{\alpha\beta})_1\right),\\
h_2&=-\frac{1}{2D}\left(\left(\frac{R_2-\bar{R}_\alpha X^\alpha_2}{2}+\bar{Ric}_\alpha(\nu,\nu) X^\alpha_2\right)\det(g_{\alpha\beta})+\left(\frac{R-\bar{R}}{2}+\bar{Ric}(\nu,\nu)\right)\det(g_{\alpha\beta})_2\right),\\
h_3&=0,\\
h_4&=\frac{\bar{R}_{\nu 221}-g^{11}\bar{Ric}(e_1,\nu)\det(g_{\alpha\beta})}{\sqrt{D}}.
\end{align*}

Similarly
\begin{align*}
\Delta y&=\frac{\nabla_2 h_{11}-\nabla_1 h_{12}}{\sqrt{D}}J(x,y)-\frac{h_{11}D_2-h_{12}D_1}{2D^{3/2}}J(x,y).
\end{align*}

\begin{align*}
\nabla_2 h_{11}-\nabla_1 h_{12}=\bar{R}_{\nu 112}.
\end{align*}

and
\begin{align*}
h_{11}D_2-h_{12}D_1=&h_{11}\left(\frac{R_2-\bar{R}_\alpha X^\alpha_2}{2}+\bar{Ric}_\alpha(\nu,\nu) X^\alpha_2+2\bar{Ric}(e_j,\nu)h^j_2\right)\det(g_{\alpha\beta})\\
&+h_{11}\left(\frac{R-\bar{R}}{2}+\bar{Ric}(\nu,\nu)\right)\det(g_{\alpha\beta})_2\\
&-h_{12}\left(\frac{R_1-\bar{R}_\alpha X^\alpha_1}{2}+\bar{Ric}_\alpha(\nu,\nu) X^\alpha_1+2\bar{Ric}(e_j,\nu)h^j_1\right)\det(g_{\alpha\beta})\\
&-h_{12}\left(\frac{R-\bar{R}}{2}+\bar{Ric}(\nu,\nu)\right)\det(g_{\alpha\beta})_1\\
=&\frac{|Dy|^2\sqrt{D}}{J(x,y)}\left(\frac{R_2-\bar{R}_\alpha X^\alpha_2}{2}+\bar{Ric}_\alpha(\nu,\nu) X^\alpha_2\right)\det(g_{\alpha\beta})\\
&+\frac{|Dy|^2\sqrt{D}}{J(x,y)}\left(\frac{R-\bar{R}}{2}+\bar{Ric}(\nu,\nu)\right)\det(g_{\alpha\beta})_2\\
&+\frac{Dx\cdot Dy \sqrt{D}}{J(x,y)}\left(\frac{R_1-\bar{R}_\alpha X^\alpha_1}{2}+\bar{Ric}_\alpha(\nu,\nu) X^\alpha_1\right)\det(g_{\alpha\beta})\\
&+\frac{Dx\cdot Dy \sqrt{D}}{J(x,y)}\left(\frac{R-\bar{R}}{2}+\bar{Ric}(\nu,\nu)\right)\det(g_{\alpha\beta})_1\\
&+2\left(h_{11}h^2_2-h_{12}h^2_1\right)\bar{Ric}(e_2,\nu)\det(g_{\alpha\beta}).
\end{align*}
thus
\begin{align*}
\tilde{h}_1&=0,\\
\tilde{h}_2&=-\frac{1}{2D}\left(\left(\frac{R_1-\bar{R}_\alpha X^\alpha_1}{2}+\bar{Ric}_\alpha(\nu,\nu) X^\alpha_1\right)\det(g_{\alpha\beta})+\left(\frac{R-\bar{R}}{2}+\bar{Ric}(\nu,\nu)\right)\det(g_{\alpha\beta})_1\right),\\
\tilde{h}_3&=-\frac{1}{2D}\left(\left(\frac{R_2-\bar{R}_\alpha X^\alpha_2}{2}+\bar{Ric}_\alpha(\nu,\nu) X^\alpha_2\right)\det(g_{\alpha\beta})+\left(\frac{R-\bar{R}}{2}+\bar{Ric}(\nu,\nu)\right)\det(g_{\alpha\beta})_2\right),\\
\tilde{h}_4&=\frac{\bar{R}_{\nu 112}-g^{22}\bar{Ric}(e_2,\nu)\det(g_{\alpha\beta})}{\sqrt{D}}.
\end{align*}
The propsition is now proved.

\end{proof}

\medskip

We now prove the bound for principal curvature and its H\"older norm. 
\begin{theo}\label{thm-curvature}
Let $(N^3,\bar{g})$ be a $C^3$ Riemannian manifold with possibly finite many boundary components, which are minimal surfaces. Suppose $\Sigma=(\mathbb{S}^2,g)$ is $C^3$ surface isometrically embedded into $(N^3,\bar{g})$, let $X$ be the embedding, assume that 
\begin{align}
R(x)-\bar{R}(X(x))+2\bar{Ric}_{X(x)}(\nu,\nu)\geq C_0>0,
\end{align}
$\forall x\in \Sigma$. Assume further that $\bar{R}\geq -6\kappa^2$ for some costant $\kappa>0$. Then we have
\begin{align*}
\|h_{ij}\|_{C^\mu}\leq C,
\end{align*}
for any $0<\mu<1$, where $C$ depends on $C_0$, $\kappa$, $\|g\|_{C^3}$ and $\|\bar{g}\|_{C^3}$, but not on the position of $\Sigma$ in $N^3$.
\end{theo}

\begin{proof}
We first verify the condition
\begin{align*}
\int_\Sigma Hd\sigma \leq M.
\end{align*}
Take $\kappa$ large enough, then we have 
\begin{align*}
R>-2\kappa^2.
\end{align*}

By Lemma \ref{quasi-2}, we have
\begin{align*}
\int_\Sigma \left(H_0-H\right)\cosh (\kappa r) d\sigma \geq 0.
\end{align*}

On the other hand, by (\ref{wyst inequality}) in the proof of Theorem \ref{thm-hyperbolic}, we have
\begin{align*}
\int_\Sigma H_0 \cosh (\kappa r) d\sigma \leq M,
\end{align*}
where $M$ only depends on $\kappa$ and $\|g\|_{C^2}$.

Thus
\begin{align*}
\int_\Sigma Hd\sigma \leq \int_\Sigma H\cosh (\kappa r) d\sigma \leq \int_\Sigma H_0\cosh (\kappa r) d\sigma\leq M.
\end{align*}

We can now apply Propsition \ref{key prop}, by which we have
\begin{align}
\Delta x&=h_1(z)|Dx|^2+h_2(z)Dx\cdot Dy+h_3(z)|Dy|^2+h_4(z)Dx\wedge Dy,\\\nonumber
\Delta y&=\tilde{h}_1(z)|Dx|^2+\tilde{h}_2(z)Dx\cdot Dy+\tilde{h}_3(z)|Dy|^2+\tilde{h}_4(z)Dx\wedge Dy.
\end{align}

Moreover, all the assumption of Lemma \ref{Heinz-Lewy theorem} is satisfied, thus Lemma \ref{Heinz-Lewy theorem} is applicable, so we have
\begin{align}\label{Jacobian}
|Dx|^2,|Dy|^2\leq C,\quad J(x,y)\geq c>0,
\end{align}
for $(x,y)\in B_{r_0/2}$. 

In view of (\ref{conformal-3}), we have
\begin{align*}
|h_{ij}|\leq C,
\end{align*}
for $(x,y)\in B_{r_0/2}$.

Similar to the proof of Theorem \ref{thm-hyperbolic}, we have
\begin{align}\label{Holder}
\|h_{ij}\|_{C^\mu}\leq C,
\end{align}
for $(x,y)\in B_{r_0/2}$.

The theorem now follows by the fact that $\Sigma$ can be covered by finite patches of $B_{r_0/2}$.

\end{proof}

We are now in position to prove Theorem \ref{thm-1}.

\begin{proof}
To begin with, by Theorem \ref{thm-curvature}, in order to prove Theorem \ref{thm-1}, we only need to study the high regularity of the embedding. 

Let's restrict ourselves in local consideration. First choose $O$ inside $\Sigma$ and sufficiently close to $\Sigma$. Take $r$ such that it is smaller or equal than the injective radius of $\Sigma$ and $N$.  Let $B_r\subset N$, a geodisic ball centered at $O$. Denote $S_r=B_r\cap \Sigma$. 

Since the second fundamental form of $\Sigma$ is positive definite, $S_r$ is in fact convex and thus star-shaped w.r.t. $O$. Moreover, we have $\left\langle\partial_r,\nu\right\rangle\geq C_0>0$, where $\nu$ is the outer unit normal.

\medskip

Since $r$ is smaller or equal than the injective radius of $\Sigma$ and $N$, we can write the metric in $B_r$ as
\begin{align*}
ds^2=dr^2+\phi_{ij}(r,\theta)\sigma_{ij},
\end{align*}
where $\sigma_{ij}$ is the standard metric on $\mathbb{S}^2$, and $\phi_{ij}$ is a function of $r,\theta$. 

Since $S_r$ is star-shaped, we can write $S_r$ as a graph on $\mathbb{S}^2$, i.e. $(\theta,r)$. We have
\begin{align*}
e_i=\partial_i+r_i\partial_r,
\end{align*}
and
\begin{align*}
g_{ij}=r_ir_j+\phi_{ij}\sigma_{ij}.
\end{align*}
We then have
\begin{align*}
\nu=\frac{1}{v}\left(-\frac{r^i}{\phi_{ii}}\partial_i+\partial_r\right),
\end{align*}
where
\begin{align*}
v^2=1+\frac{r_i^2}{\phi_{ii}}.
\end{align*}
Now
\begin{align*}
-h_{ij}&=\left\langle\bar{\nabla}_{e_j}e_i,\nu\right\rangle=\frac{1}{v}\left\langle r_j\bar{\nabla}_{\partial_r}\left(\partial_i+r_i\partial_r\right)+\bar{\nabla}_{\partial_j}\left(\partial_i+r_i\partial_r\right), -\frac{r^k}{\phi_{kk}}\partial_k+\partial_r\right\rangle.
\end{align*}

Since $\partial_r$ is the geodesic vector field, i.e. $\bar{\nabla}_{\partial_r}\partial_r=0$, we have
\begin{align*}
-h_{ij}=&\frac{1}{v}\left\langle r_j\bar{\nabla}_{\partial_r}\partial_i+\bar{\nabla}_{\partial_j}\partial_i++r_i\bar{\nabla}_{\partial_j}\partial_r+r_{ij}\partial_r, -\frac{r^k}{\phi_{kk}}\partial_k+\partial_r\right\rangle\\
=&\frac{1}{v}\left(r_j\left\langle\bar{\nabla}_{\partial_r}\partial_i,\partial_r\right\rangle+\left\langle\bar{\nabla}_{\partial_j}\partial_i,\partial_r\right\rangle+r_{ij}\right)\\
&-\frac{r^k}{v\phi_{kk}}\left(r_j\left\langle\bar{\nabla}_{\partial_r}\partial_i,\partial_k\right\rangle+\left\langle\bar{\nabla}_{\partial_j}\partial_i,\partial_k\right\rangle+r_i\left\langle\bar{\nabla}_{\partial_j}\partial_r,\partial_k\right\rangle\right)\\
=&\frac{1}{v}\left(\left\langle\bar{\nabla}_{\partial_j}\partial_i,\partial_r\right\rangle+r_{ij}\right)-\frac{r^k}{v\phi_{kk}}\left(-r_j\left\langle\bar{\nabla}_{\partial_i}\partial_k,\partial_r\right\rangle+\left\langle\bar{\nabla}_{\partial_j}\partial_i,\partial_k\right\rangle-r_i\left\langle\partial_r,\bar{\nabla}_{\partial_j}\partial_k\right\rangle\right).\\
\end{align*}

Recall that
\begin{align*}
\Gamma^m_{ij}=\frac{1}{2}\sum_k\left(g_{jk,i}+g_{ki,j}-g_{ij,k}\right)g^{km}.
\end{align*}
It follows
\begin{align*}
\Gamma^r_{ij}=-\frac{\partial\phi_{ij}}{\partial r}\frac{\sigma_{ij}}{2}.
\end{align*}

Thus
\begin{align*}
-h_{ij}&=\frac{1}{v}\left(-\frac{\partial\phi_{ij}}{\partial r}\frac{\sigma_{ij}}{2}+r_{ij}-\frac{r_ir_j}{2\phi_{ii}}\frac{\partial\phi_{ii}}{\partial r}-\frac{r_ir_j}{2\phi_{jj}}\frac{\partial\phi_{jj}}{\partial r}-r^k\Gamma^k_{ij}\right)\\
&=\frac{1}{v}\left(r_{i,j}-\frac{\partial\phi_{ij}}{\partial r}\frac{\sigma_{ij}}{2}-\frac{r_ir_j}{2\phi_{ii}}\frac{\partial\phi_{ii}}{\partial r}-\frac{r_ir_j}{2\phi_{jj}}\frac{\partial\phi_{jj}}{\partial r}\right).
\end{align*}

We first note that $\|r\|_{C^1}\leq C$, which can be deduced by the construction and the metric. By the above representation and Theorem \ref{thm-curvature}, we have
\begin{align*}
\|r\|_{C^{2,\mu}}\leq C,
\end{align*}
for any $0<\mu<1$.

For high regularity, by Gauss equation, we have
\begin{align*}
\det(h_{ij})=\det(g_{ij})\left(\frac{R-\bar{R}}{2}+\bar{Ric}(\nu,\nu)\right).
\end{align*}

This is a Monge-Amp\`ere type equation with right hand side positive. Since $r\in C^{2,\mu}$, we can apply Schauder estimate to obtain high regularity. Theorem \ref{thm-1} is thus proved.

\end{proof}

\section{Isometric embedding into Riemannian manifold}
In this section, we prove the isometric embedding into general Riemannian manifold. We use the continuity method to solve the problem. In view of the continuity method, our main result, Theorem \ref{thm-1} serves as the closedness part, moreover, since the estimate is independent of the position, it also ensures the uniform ellipticity of the linearized equation. The openness part is done in a recent work by Li and Wang \cite{GLW}.

We now state two theorems by Li and Wang \cite{GLW}.
\begin{theo}\label{linearized theorem}
Suppose $X$ is a strictly convex surface $(\mathbb{S}^2,g)$ in general Riemannian manifold $(N^3,\bar{g})$. Then, for any symmetric two tensor $q$, there exists some vector field $\tau$ satisfying the linearized equation,
\begin{align}\label{linearized equation}
dX\cdot D\tau=q.
\end{align}
Moreover, for $\tau$ perpendicular to the kernel of (\ref{linearized equation}), 
\begin{align*}
\|\tau\|_{C^0}\leq C,
\end{align*}
where $C$ depends on the upper and lower bound of principal curvature and $\|q\|_{C^{1,\alpha}}$, $0<\alpha<1$.
\end{theo}

\begin{theo}\label{openness theorem}
Suppose $(\mathbb{S}^2,g)$ can be isometrically embedded into a general Riemannian manifold $(N^3,\bar{g})$ as a closed strictly convex surface. Then for any $\alpha\in (0,1)$, there exists a positive $\epsilon$, depending only on $C_0$, $\|g\|_{C^3}$, $\|\tilde{g}\|_{C^3}$ $\|\bar g\|_{C^3}$ and $\alpha$, such that for any smooth metric $\tilde{g}$ on $\mathbb{S}^2$ satisfying
\begin{align*}
\|g-\tilde{g}\|_{C^{2,\alpha}}<\epsilon,
\end{align*}
$(\mathbb{S}^2,\tilde{g})$ can also be isometrically embedded into $(N,\bar{g})$ as another closed strictly convex surface.
\end{theo}

We now prove Theorem \ref{thm-isometric}.
\begin{proof}
To solve the problem, we indeed need to solve the following equation
\begin{align*}
dX\cdot dX=g.
\end{align*}
By continuity method, we need to solve 
\begin{align*}
dX_t\cdot dX_t=g_t,
\end{align*}
for $t\in [0,1]$. The linearization equation is
\begin{align*}
dX_t\cdot D\tau_t=q_t,
\end{align*}
where $\tau_t$ is the variation vector field and $D$ is the connection of $(N,\bar{g})$. 

The homotopic path $g_t$ consists of three parts. For $t=0$, we can choose $X_0$ as a geodesic sphere of sufficient small radius centered at a point far away from the compact boundary. It's clear that $X_0$ is of sufficient large scalar curvature. We choose the normalized Ricci flow of $g_0$ as the first part of homotopic path as in \cite{WL}. At the end of this part, we have a metric $g_{\frac{1}{3}}$ of constant scalar curvature. The third part is the normalized Ricci flow of $g$, at the end of this part, we have a metric $g_{\frac{2}{3}}$ of constant scalar curvature. Now we embed both $(\mathbb{S}^2,g_{\frac{1}{3}})$ and $(\mathbb{S}^2,g_{\frac{2}{3}})$ into hyperbolic space. The last part of the homotopic path is to shrink $(\mathbb{S}^2,g_{\frac{2}{3}})$ to $(\mathbb{S}^2,g_{\frac{1}{3}})$ in hyperbolic space. Clearly, from our construction, the condition (\ref{curvature con thm}) always holds.

Now consider the closedness part, by Theorem \ref{thm-curvature}, the mean curvature estimate is independent of the position $(\mathbb{S}^2,g)$ in $(N,\bar{g})$, thus the equation system (\ref{linearized equation}) is uniformly elliptic, Theorem \ref{linearized theorem} is then applicable, which yields the bound of $\tau$. The uniform $C^0$ estimate follows by the formula
\begin{align*}
X(x)=X_0(x)+\int_0^1\tau_t(x)dt.
\end{align*}
Moreover, by our choice of $X_0$, $X$ will stay away from the compact boundary along the homotopic path. Now $C^1$ estimate follows from the metric, $C^2$ and high regularity follows from Theorem \ref{thm-1}. The openness part follows from Theorem \ref{openness theorem}. This fulfills the continuity method and thus the theorem is proved.
\end{proof}

\section{Appendix}
To prove Lemma \ref{Heinz-Lewy Dini}, we follow the steps in the proof of Theorem 8.3.2 in \cite{S}. Since most of the procedures are the same, we only need to establish three lemmas that are different.

The first one is $C^1$ estimate for divergence elliptic equations with Dini coefficients. This was recently proved by Y.Y. Li in \cite{Li}. The following lemma is only a special case in \cite{Li}.
\begin{lemm}\label{Li}
Let $u\in W^{1,2}(B_4)$ be a solution of the following elliptic equation,
\begin{align*}
D_i(A_{ij}(x)D_j u)=0,\quad i,j=1,\cdots,n,
\end{align*}
where $B_4$ is the ball in $\mathbb{R}^n$ of radius $4$ at the origin. Assume that $A_{ij}$ is uniformly elliptic and Dini continuous in $B_4$. Then $ u\in C^1_{B_1}$.
\end{lemm}

\begin{rema}
Our assumption is slightly different from Lemma \ref{Li} above as $a(z,w)$ depends on $z,w$. However, since $a$ is continuous, we immediately get $w\in C^\mu$ for any $0<\mu<1$. Now $a$ can be considered as a function of $z$, and is still Dini continuous as being H\"older continuous is stronger than being Dini continuous.
\end{rema}

The second one is the replacement of Theorem 7.3.1 in \cite{S}, which was again proved by Schulz in \cite{S1}.
\begin{lemm}
Let $\varphi(z)\in C^1(\Omega)$ satisfy the elliptic equation
\begin{align*}
D_i(a(z)D_i\varphi)=0,
\end{align*}
such that $\lambda \leq a(z)\leq \Lambda $ and $a\in L^\infty(\Omega)$. If $\varphi=o(|z|^n)$ as $|z|\rightarrow 0$ for all $n\in \mathbb{N}$, then $\varphi(z)=0$.
\end{lemm}

The third lemma is the repalcement of Theorem 7.2.1 in \cite{S}.
\begin{lemm}
Let $\varphi(z)\in C^1(\Omega)$ satisfy the elliptic equation
\begin{align*}
D_i(a(z)D_i\varphi)=0,
\end{align*}
such that $\lambda \leq a(z)\leq \Lambda $ and $a$ is Dini continuous. If $\varphi(z)=o(|z|^n)$ for some $n\in \mathbb{N}_0$, then
\begin{align*}
\lim_{|z|\rightarrow 0,z\neq 0}\frac{\varphi_z(z)}{z^n}
\end{align*}
exists.
\end{lemm}

\begin{proof}
To begin with, we assume that $\varphi$ is not identically $0$, otherwise there's nothing to prove. Following the steps in \cite{AM} and \cite{S1}, we consider another function satisfying
\begin{align*}
\phi_x=-\frac{a}{a(0)}\varphi_y,\quad \phi_y=\frac{a}{a(0)}\varphi_x.
\end{align*}
Let $\eta=\varphi+i\phi$, we have
\begin{align*}
\eta_{\bar{z}}=\frac{a(0)-a}{a(0)+a}\bar{\eta}_{\bar{z}}.
\end{align*}

By Bers-Nirenberg's representation theorem \cite{BN}, we have
\begin{align*}
\eta=f(\chi),
\end{align*}
where $f$ is an analytic function, $\chi$ is a homeomorphism satisfying
\begin{align*}
\chi_{\bar{z}}=\frac{a(0)-a}{a(0)+a} \frac{\bar{\eta}_{\bar{z}}}{\eta_z}\chi_z=\frac{a(0)-a}{a(0)+a} \frac{\overline{f_z(\chi)}}{f_z(\chi)}\bar{\chi}_{\bar{z}},
\end{align*}
with $\chi,\chi^{-1}$ being H\"older continuous and $\chi(0)=0$. 

Thus we have 
\begin{align*}
|\chi(z_1)-\chi(z_2)|\leq C|z_1-z_2|^\alpha,
\end{align*}
and
\begin{align*}
|\chi(z)|\geq C|z|^{\frac{1}{\alpha}},
\end{align*}  
for some $0<\alpha<1$.

Assume that
\begin{align*}
f_z(z)=a_k z^k+a_{k+1}z^{k+1}+\cdots,
\end{align*}
where $a_k\neq 0$ and $k\geq 1$. Note that $k\neq 0$, for otherwise $\eta$ would automatically be a local homeomorphism by the representation above, which is not true by the assumption.

Rewrite the above equation as
\begin{align*}
\chi_{\bar{z}}=Ae^{i\theta}\bar{\chi}_{\bar{z}},
\end{align*}
where $A=\frac{a(0)-a}{a(0)+a}$ and $e^{i\theta}=\frac{\overline{f_z(\chi)}}{f_z(\chi)}$.

We now show that $Ae^{i\theta}$ is Dini continuous. To show that, let $z_1\neq z_2$, if $z_1=0$ or $z_2=0$, then it's trivial. Suppose now $z_1,z_2\neq 0$, such that $|z_1|\leq |z_2|$, we split into two cases. Let us pick $\beta>0$ such that $\alpha-\frac{1}{\alpha \beta}>0$, e.g., $\beta=\frac{2}{\alpha^2}$. 

Case 1, $|z_1-z_2|\leq |z_1|^\beta$, for $\alpha-\frac{1}{\alpha\beta}>0$. We first show that
\begin{align*}
|\frac{\chi(z_2)}{\chi(z_1)}|\leq 1+|\frac{\chi(z_2)-\chi(z_1)}{\chi(z_1)}|\leq 1+C\frac{|z_1-z_2|^\alpha}{|z_1|^{\frac{1}{\alpha}}}\leq 1+C|z_1-z_2|^{\alpha-\frac{1}{\alpha\beta}}\leq 2,
\end{align*}
for $|z_1-z_2|$ sufficiently small.

Now we have 
\begin{align*}
&|\frac{\overline{\chi}^k(z_1)}{\chi^k(z_1)}-\frac{\overline{\chi}^k(z_2)}{\chi^k(z_2)}|\leq C|\frac{\overline{\chi}^k(z_1)-\overline{\chi}^k(z_2)}{\chi^k(z_1)}|+C|\frac{\overline{\chi}^k(z_2)(\chi^k(z_2)-\chi^k(z_1))}{\chi^k(z_1)\chi^k(z_2)}|\\
&\leq C\frac{|\chi^k(z_1)-\chi^k(z_2)|}{|\chi(z_1)|^k}\leq C\frac{|\chi(z_1)-\chi(z_2)|}{\chi(z_1)}\leq C|z_1-z_2|^{\alpha}|z_1|^{-\frac{1}{\alpha}}\leq C|z_1-z_2|^{\alpha-\frac{1}{\alpha\beta}}.
\end{align*}
we have used the fact $|\frac{\chi(z_2)}{\chi(z_1)}|\leq 2$ in the inequality.

Since
\begin{align*}
e^{i\theta}=\frac{\overline{\chi}^k}{\chi^k}\frac{\bar{a}_k+\bar{a}_{k+1}\overline{\chi}+\cdots}{a_k+a_{k+1}\chi+\cdots}.
\end{align*}

We have
\begin{align*}
|e^{i\theta}(z_1)-e^{i\theta}(z_2)|\leq C|z_1-z_2|^\alpha+C|z_1-z_2|^{\alpha-\frac{1}{\alpha\beta}}.
\end{align*}

Thus
\begin{align*}
|Ae^{i\theta}(z_1)-Ae^{i\theta}(z_2)|&\leq |A(z_1)||e^{i\theta}(z_1)-e^{i\theta}(z_2)|+|A(z_1)-A(z_2)||e^{i\theta}(z_2)|\\
&\leq C|z_1-z_2|^\alpha+C|z_1-z_2|^{\alpha-\frac{1}{\alpha\beta}}+C\omega(|z_1-z_2|),
\end{align*}
where $\omega$ is the moduli of continuity of $a$.

Case 2, $|z_1-z_2|\geq |z_1|^\beta$. We have
\begin{align*}
|Ae^{i\theta}(z_1)-Ae^{i\theta}(z_2)|&\leq |A(z_1)||e^{i\theta}(z_1)-e^{i\theta}(z_2)|+|A(z_1)-A(z_2)||e^{i\theta}(z_2)|\\
&\leq C\omega(|z_1-z_2|^{\frac{1}{\beta}})+C\omega(|z_1-z_2|).
\end{align*}

To sum up, we have
\begin{align*}
|Ae^{i\theta}(z_1)-Ae^{i\theta}(z_2)| \leq C|z_1-z_2|^\alpha+C|z_1-z_2|^{\alpha-\frac{1}{\alpha\beta}}+C\omega(|z_1-z_2|^{\frac{1}{\beta}})+C\omega(|z_1-z_2|).
\end{align*}

Let $\omega_0$ be moduli of continuity of $Ae^{i\theta}$ then
\begin{align*}
\omega_0(r)\leq Cr^\alpha+ Cr^{\alpha-\frac{1}{\alpha\beta}}+C\omega(r^{\frac{1}{\beta}})+C\omega(r).
\end{align*}
Since $\omega(r)$ satisfies Dini condition implies that $\omega(r^{\frac{1}{\beta}})$ also satisfies Dini condition, we have $Ae^{i\theta}$ is Dini continuous.

Let $\chi=u+iv$, then we have
\begin{align*}
u_x-v_y+i(v_x+u_y)=Ae^{i\theta}(u_x+v_y+i(u_y-v_x)),
\end{align*}
i.e.
\begin{align*}
u_x-v_y&=A\cos\theta(u_x+v_y)-A\sin\theta(u_y-v_x),\\
v_x+u_y&=A\cos\theta(u_y-v_x)+A\sin\theta(u_x+v_y).
\end{align*}
equivalently
\begin{align*}
(1-A\cos\theta)u_x+A\sin\theta u_y&=A\sin\theta v_x+(1+A\cos\theta)v_y,\\
-A\sin\theta u_x+(1-A\cos\theta)u_y&=-(1+A\cos\theta)v_x+A\sin\theta v_y.
\end{align*}
thus
\begin{align*}
(1+A^2-2A\cos\theta)u_x&=2A\sin\theta v_x+(1-A^2)v_y,\\
(1+A^2-2A\cos\theta)u_y&=-(1-A^2) v_x+2A\sin\theta v_y.
\end{align*}
i.e. 
\begin{align*}
D_i\left(A_{ij}v_j\right)=0,
\end{align*}
where
\begin{align*}
A_{11}=A_{22}=\frac{1-A^2}{1+A^2-2A\cos\theta},\quad A_{12}=-A_{21}=\frac{2A\sin\theta}{1+A^2-2A\cos\theta}.
\end{align*}
Since $Ae^{i\theta}$ is Dini continuous, thus $A_{ij}$ is Dini continuous. On the other hand, $A(z)$ is sufficiently small as $z$ tends to $0$, thus $A_{ij}$ is uniformly elliptic. Since $\chi$ is quasiconformal mapping, $W^{1,2}$ estimate of $v$ follows by classical results (e.g. Lemma 12.1 in \cite{GT}). We can now apply Lemma \ref{Li} to obtain that $v\in C^1_{loc}$.

Similarly,
\begin{align*}
(1+A^2+2A\cos\theta)v_x&=2A\sin\theta u_x-(1-A^2)v_y,\\
(1+A^2+2A\cos\theta)v_y&=(1-A^2)u_x+2A\sin\theta v_y.
\end{align*}
i.e.
\begin{align*}
D_i\left(B_{ij}u_j\right)=0,
\end{align*}
where
\begin{align*}
B_{11}=B_{22}=\frac{1-A^2}{1+A^2+2A\cos\theta},\quad B_{12}=-B_{21}=\frac{2A\sin\theta}{1+A^2+2A\cos\theta}.
\end{align*}
and again $u\in C^1_{loc}$.

Thus we conclude that $\chi\in C^1_{loc}$.

Since $\chi$ is a homeomorphism and $\chi_{\bar{z}}(0)=0$, thus we have
\begin{align*}
\chi=C_0z+o(|z|),\quad \chi_z=C_0+o(1),
\end{align*}
for $C_0\neq 0$.
By the representation, the lemma is now proved.

\end{proof}

\medskip

\noindent
{\it Acknowledgement}: The author would like to express gratitude to his supervisor Professor Pengfei Guan for suggesting the problem, helpful discussions and consistent support, without which the work can never be done. He would also thank Professor Pengzi Miao for kind remark and bringing Lemma \ref{quasi-2} to his attention.

\end{document}